\documentclass[11pt, letterpaper, english]{amsart}
\usepackage{amsmath} 
\usepackage{amsthm} 
\usepackage{amssymb} 
\usepackage[dvipsnames]{xcolor}
\usepackage{amscd} 
\usepackage{mathtools}
\usepackage{booktabs}
\usepackage[boxed]{algorithm2e}
\usepackage{subcaption}

\usepackage{array} 
\usepackage[cal=boondoxo,scr=euler]{mathalfa}
\usepackage[backref=page,linktocpage]{hyperref} 
\usepackage{cleveref} 
\usepackage{caption} %
\usepackage{graphics,graphicx} 

\usepackage{enumerate} 

\DeclareMathAlphabet{\mathsf}{OT1}{\sfdefault}{m}{n}

\newcommand{\nocontentsline}[3]{}
\newcommand{\tocless}[2]{\bgroup\let\addcontentsline=\nocontentsline#1{#2}\egroup}

\hypersetup{
    colorlinks = true,
    linkbordercolor = {white},
    linkcolor = {NavyBlue},
    anchorcolor = {black},
    citecolor = {NavyBlue},
    filecolor = {cyan},
    menucolor = {NavyBlue},
    runcolor = {cyan},
    urlcolor = {NavyBlue}
}

\usepackage[margin=1.65in]{geometry} 

\usepackage{verbatim}

\usepackage{scalerel}

\newtheoremstyle{teoremas}
{12pt}
{13pt}
{\itshape}
{}
{\bfseries}
{}
{.5em}
{}

\theoremstyle{teoremas}
\newtheorem{theorem}{Theorem}[section]

\newtheorem{proposition}[theorem]{Proposition}

\newtheoremstyle{definition}
{12pt}
{12pt}
{}
{}
{\bfseries}
{}
{.5em}
{}

\theoremstyle{definition}

\newtheorem{conjecture}[theorem]{Conjecture}

\title{Almost factorial many facets for $0/1$-polytopes}

\author{Federico Castillo}
\address{Centro de Investigaci\'on en Matem\'aticas, Guanajuato}
\email{federico.castillo@cimat.mx}

\author{Luis Ferroni}
\address{Dipartimento di Matematica, Universit\`a di Pisa}
\email{luis.ferroni@unipi.it}

\begin{document}

\begin{abstract}
    A long-standing question posed by Fukuda (1995) and Ziegler (2000) inquires about the asymptotic behavior of $g(n)$, the maximum number of facets that an $n$-dimensional $0/1$-polytope can have. A remarkable result by B\'ar\'any and P\'or (2001) via probabilistic methods established that $g(n)$ is at least superexponential in $n$. In this paper, we propose a drastic change of perspective, which leads us to show that for each $n\geq 10$ there exists a $0/1$-polytope having at least $(n-\lceil 2\log_2 (n)\rceil - 1)!$ facets. This provides a very significant improvement over the currently known lower bounds for $g(n)$. Furthermore, when combined with known upper bounds, our construction establishes the correct asymptotic behavior of $\log g(n)$ up to an error of $O((\log n)^2)$. The methods employed throughout this paper are elementary and fully deterministic. The underlying ideas in our proof stem from the combinatorics of hypersimplices and permutohedra.
\end{abstract}

\maketitle

\section{Introduction}

A recurring theme in mathematics, especially within combinatorial optimization and adjacent areas, consists of estimating the number of facets of certain classes of convex polyhedra. 
Well-known tools to approach linear-programming problems rely on an outer description of a polyhedron, and the number of facets of this object imposes a lower bound on the number of inequalities needed to describe it. 

An important class of convex polytopes that pervades mathematics is that of $0/1$-polytopes, i.e., those having vertices for which every coordinate is equal to $0$ or $1$.
We warmly recommend \cite{ziegler2000lectures} for a survey of this class.
Even though their vertex description is simple, these polytopes exhibit complex behavior; for instance, the entries of the integer facet normals can be huge \cite{alon1997anti}.

A fundamental question posed over three decades ago by Fukuda (at Oberwolfach in 1995) and Ziegler~\cite{ziegler2000lectures} asks for the asymptotic behavior of the following quantity:
    \[ g(n) := \max \{\text{$f_{n-1}(\mathcal{P})$} : \text{$\mathcal{P}\subseteq \mathbb{R}^n$ full-dimensional $0/1$-polytope}\},\]
where $f_i(\mathcal{P})$ stands for the number of $i$-dimensional faces of $\mathcal{P}$. 

Kortenkamp, Richter-Gebert, Sarangarajan, and Ziegler studied this question systematically in \cite{kortenkamp1997extremal}, where they surveyed and improved what was known at the time.

The upper bound shown in \cite{kortenkamp1997extremal} is 
\(g(n) \leq n! - (n-1)! + 2(n-1),\) and its proof relied on an approach attributed to B\'ar\'any.
This was improved by Fleiner--Kaibel--Rote \cite{fleiner2000upper}, who proved that \begin{equation}\label{eq:upper-bound}
    g(n) \leq c(n-2)!
\end{equation}
for some absolute constant $c\leq 30$ and all $n$ sufficiently large. The lower bound presented in \cite{kortenkamp1997extremal} is exponential in $n$.
Soon after, there was a giant leap by B\'ar\'any and P\'or \cite{barany-por} showing that there exists an absolute constant $c>0$ such that
    \begin{equation} \label{eq:bound-barany-por}
    g(n) \geq \left(\frac{cn}{\log n}\right)^{n/4}.
    \end{equation}
Remarkably, B\'ar\'any and P\'or's proof is non-constructive. 
They study $0/1$-polytopes by choosing vertices uniformly at random and then bound the expected value of the number of facets.

It is worth mentioning that over the last three decades there have been improvements to the lower bound in equation~\eqref{eq:bound-barany-por} that refine the approach of \cite{barany-por}.
The work by Gatzouras, Giannopoulos, and Markoulakis \cite{gatzouras-giannopoulos-markoulakis1,gatzouras-giannopoulos-markoulakis2} improved \eqref{eq:bound-barany-por} by doubling the exponent, i.e., they showed the existence of another absolute constant $c>0$ such that 
    \begin{equation} \label{eq:bound-ggm}
    g(n) \geq \left(\frac{cn}{\log n}\right)^{n/2}.
    \end{equation}
Very recently, Friedland \cite{friedland} pushed this further by showing that the logarithmic denominator can be suppressed, i.e., he showed the existence of yet another absolute constant $c>0$ for which
    \begin{equation} \label{eq:bound-friedland}
    g(n) \geq \left(cn\right)^{n/2}.
    \end{equation}
In the introduction of \cite{barany-por}, B\'ar\'any and P\'or offer an analogy with the problem of bounding the expected number of facets of an $n$-dimensional polytope defined as the convex hull of $N$ points on the unit sphere, a question that has been studied in \cite{buchta1984random, buchta1985stochastical}.
If the class of $n$-dimensional $0/1$-polytopes were to behave similarly, then there would be an upper bound of the form $(cn)^{n/2}$, where $c$ is a constant.
Under that assumption, the bound in equation \eqref{eq:bound-friedland} would determine the exact growth rate of $\log g(n)$. Our result unveils that the analogy is not accurate:
we show that $\log g(n)$ is asymptotic to $n\log n$, rather than to $n\log n/2$ as suggested.
In particular, no upper bound of the form $(cn)^{n/2}$ can exist.


\begin{theorem}\label{thm:main}
    For every $n\geq 10$ there exists an $n$-dimensional $0/1$-polytope $\mathcal{P}\subseteq \mathbb{R}^n$ whose number of facets is at least $(n-\lceil 2\log_2 n\rceil - 1)!$.
\end{theorem}

The polytopes we use to achieve this bound have explicit coordinates for their vertices.
Our method is constructive and quite elementary, using nothing but basic linear algebra.

In contrast with the probabilistic approach, our approach is informed by the combinatorics of hypersimplices, a class of polytopes that plays a prominent role in matroid theory.
The starting point was the second author's previous work with Schr\"oter \cite{ferroni2025face}, which treats the computation of $f$-vectors of a very special class of $0/1$-polytopes called \emph{split matroid polytopes} (see \cite{joswig-schroter,ferroni-schroter}). These matroid polytopes are obtained by splitting a hypersimplex along hyperplanes that do not intersect in its interior: this procedure allows one to create matroid polytopes with many facets, but their number is still far from matching the bounds in \eqref{eq:bound-barany-por}, \eqref{eq:bound-ggm}, and \eqref{eq:bound-friedland}. We were initially interested in non-matroidal splits of a hypersimplex that could potentially achieve a higher number of facets. Although this strategy did not work, it was essential in leading us to the right ideas, as we explain below.

Another, admittedly rough, insight comes from the regular $n$-permutohedron
\[
\Pi_n = \operatorname{ConvexHull} \left\{ \left(\pi(1), \dots, \pi(n)\right) ~:~ \pi\in \mathfrak{S}_n \right\}.
\]
Its facet normals are $0/1$ vectors, and it has $n!$ vertices. Naively, its polar dual polytope (which is not defined in $\mathbb{R}^n$, but let us pretend for a moment that it is) ought to be a $0/1$-polytope with $n!$ facets. One crucial geometric idea underlying our construction is precisely that of preserving the incidence between vertices and facets of the permutohedron.

The natural question is how one can come so close to the factorial upper bound in \eqref{eq:upper-bound}. Note that the unit cube $[0,1]^n$ can be layered by hypersimplices $\Delta_{k,n}$ for $k=0,\ldots,n$ (these are the base polytopes of uniform matroids of rank $k$ on $n$ elements). Any permutation $\sigma\in \mathfrak{S}_n$ determines a sequence of $n+1$ points consisting of exactly one vertex of each of these hypersimplices, and in turn the convex hull of these $n+1$ points yields a maximal simplex in a staircase triangulation of the unit cube. By lifting each hypersimplex $\Delta_{k,n}$ to a height $h_k$, and by choosing the values of the $h_i$'s adequately, one can turn each lifted simplex into a facet. This phenomenon is also reflected in the fact that the Minkowski sum
\[
\Delta_{0,n} + \Delta_{1,n} + \cdots + \Delta_{n,n}
\]
equals the $n$-th permutohedron. Although this naive lift does not quite produce a $0/1$-point configuration, as the reader will find out, with some additional effort we can turn it into one.

\subsection*{Declaration of AI usage}

This paper benefited from substantial use of ChatGPT~Pro~5.6. 
The crucial idea of using the hypersimplex layers of the cube and exploiting the combinatorics of the incidence between vertices and facets of permutohedra was suggested (motivated by human geometric intuition) to ChatGPT by the authors. Subsequently, several subtle details of the construction were figured out by ChatGPT and the authors together: about fifteen guided prompts were needed before a first correct raw proof was completed. This proof was later severely simplified and written up by the authors.
Not a single sentence in this manuscript was generated by any Artificial Intelligence tool.

\subsection*{Acknowledgments}

Federico Castillo was partially supported by ANID Fondecyt Regular Project N°1260970. 
Luis Ferroni is a member of the GNSAGA group of the Istituto Nazionale di Alta Matematica (INdAM).

\section{Preliminaries}

In this paper, the set $ \mathbb{N} $ of natural numbers includes zero.
We denote the set $ \left\{ 1, \dots, n \right\} $ by $ [n] $ .
Throughout, $\log$ denotes the natural logarithm, whereas $\log_2$ denotes the logarithm to base $2$.

Let $ \mathbf{t} = \left\{ t_{n} \right\}_{n\in \mathbb{N}} $ be the sequence of triangular numbers, i.e., $ t_{n} \coloneqq 0 + 1 + \dots + n = n(n+1)/2 $.
Note that $ t_{0} = 0 $.
We define another sequence $ \mathbf{s} = \left\{ s_{n} \right\}_{n \geq 1} $ of \emph{super triangular} numbers using a greedy process.
Set $ s_{1} = 1 $, and then recursively let $ s_{n+1} $ be the largest triangular number not larger than the partial sum $ S_n \coloneqq s_1 + \cdots + s_{n} $.
The first seven terms of the sequence $ \mathbf{s} $ are $ 1,1,1,3,6,10,21 $, and the respective partial sums are $1,2,3,6,12,22,43$.

We say $ m \in \mathbb{N} $ is \emph{representable} by $ \mathbf{s} $ if there exists $ \mathrm{A} \subseteq \mathbb{N}_{>0} $ with
$ \sum_{a \in \mathrm{A}} s_a = m$.

\begin{proposition}\label{prop:saturated}
	Let $ m \geq 1 $.
	Every natural number less than or equal to $ S_{m}$ is representable by $ \mathbf{s} $.
	In particular, the super triangular number $ s_{m+1} $ can be represented using earlier super triangular numbers.
\end{proposition}

\begin{proof}
We argue by induction on $ m $, the base case $ m = 1 $ being clear.
Assume the statement for $ m = k $ and consider a natural number $ N $.
There are two cases.
\begin{enumerate}
  \item 
		If $ N \leq S_{k} $ then $ N $ is already representable using a subset of $ [k] $ by the induction hypothesis.
  \item
		If $ S_k < N \leq S_{k+1} $, then $0\leq N-s_{k+1}\leq S_k$.
		We use the previous case for $N-s_{k+1}$ and then add $s_{k+1}$.
\qedhere
\end{enumerate}
\end{proof}

The proof of Proposition \ref{prop:saturated} did not use the greedy definition of the sequence $ \mathbf{s}$.
Indeed, the same conclusion holds for the sequence of triangular numbers (after prepending two extra 1s at the beginning).
However, the partial sums of the triangular sequence grow cubically (since the triangular numbers themselves grow quadratically).
The greedy nature of the sequence $ \mathbf{s} $ produces partial sums that grow exponentially.

\begin{proposition}\label{prop:bound}
For every positive integer $ m $ we have that $ S_{m} \geq  \frac{1}{4} 2^m$.
\end{proposition}

\begin{proof}

Let $ N \in \mathbb{N}$.
If $ t_{k} $ is the largest triangular number less than or equal to $ N $, then $ t_k \leq N < t_{k+1} = t_{k} + k + 1 $.
The first inequality implies that $ k \leq \sqrt{2N} $, which combined with the second inequality gives $ N - \sqrt{2N} - 1 < t_{k} $.
Setting $ N = S_m $ and adding $ S_{m} $ on both sides we obtain 
\begin{equation}\label{eq:ineq}
     2S_{m} - \sqrt{2S_{m}} - 1 < S_{m+1},
\end{equation}
since by definition $ S_{m} + t_{k} = S_{m+1}$.
As $t_k \leq S_m$, it also follows that $S_{m+1} \leq 2S_{m}$, so $S_m \leq 2^m$ for all $m$.

Let $q_m = S_m/2^m$.
We have that $q_m \leq 1$ and we aim to show that $q_m > 1/4$.
Dividing by $2^{m+1}$ in \eqref{eq:ineq}, we arrive at 
\[
q_{m+1} - q_m > -2^{-\frac{m+1}{2}}\sqrt{q_m} - 2^{-(m+1)} > -2^{-\frac{m+1}{2}} - 2^{-(m+1)} .
\]
We have that $S_{12} = 1205$.
For $m > 12$ we add the consecutive differences to obtain
\[
q_{m} > q_{12} - \sum_{j=13}^\infty 2^{-\frac{j}{2}} - \sum_{i=13}^\infty  2^{-i} = \frac{1205}{4096} - \frac{1+\sqrt{2}}{64} - \frac{1}{4096} >\frac{1}{4}.
\]
For $m \leq 12$ the inequality can be directly checked.
\end{proof}

\section{The construction}

We work over the vector space $ \mathbb{R}^{m} \oplus \mathbb{R}^{p} $ with canonical bases $ \left\{ \mathbf{e}_{1}, \dots, \mathbf{e}_{m} \right\} $ and $ \left\{ \mathbf{f}_{1}, \dots, \mathbf{f}_{p} \right\} $ respectively.
We treat both factors differently; we think of the second factor as auxiliary.
For any subset $ \mathrm{A} \subseteq [m] $, we define $ \mathbf{e}_{\mathrm{A}} \coloneqq  \sum_{a\in \mathrm{A}} \mathbf{e}_{a} $.

We choose $ p $ to be the smallest positive integer such that $ S_p \geq t_m=m(m+1)/2 $.
Notice that we must have that $ s_{p} \leq t_m $.
By Proposition \ref{prop:bound} we have that
\begin{equation}\label{eq:p_bound}
	p \leq \left\lceil\log_2\bigl(2m(m+1)\bigr)\right\rceil < 2\log_2(m+1) + 2.
\end{equation}

The choice of $ p $ together with Proposition \ref{prop:saturated} guarantees that every triangular number up to $ t_{m} $ is representable by $ \{ s_1, \dots, s_p\}$.
This allows us to define the following vectors in $ \mathbb{R}^{p} $:
For every $ 0 \leq i \leq m $, we define $ \mathrm{R}_{i} $ to be a set representing $ t_{i} $.
If $ t_{i} = s_{j} $ is super triangular for $j\geq 2$, we choose $ \mathrm{R}_{i} $ to be contained in $ [j-1] $, which is possible by Proposition \ref{prop:saturated}.

For each $j$, let $k_j$ satisfy $t_{k_j}=s_j$. Consider the following sets of $ 0/1 $ vectors:
\begin{align*}
	\mathscr{V}_{m} & \coloneqq \left\{ \mathbf{e}_{\mathrm{S}}+\mathbf{f}_{\mathrm{R}_{k}}  \;\middle|\; \mathrm{S} \subseteq [m],\ \#\mathrm{S}=k \right\}, \\
	\mathscr{W}_{m} & \coloneqq \left\{ \mathbf{e}_{\mathrm{T}}+\mathbf{f}_{j}  \;\middle|\; \mathrm{T}\subseteq[m],\ \#\mathrm{T}=k_j,\ 2\leq j\leq p \right\}.
\end{align*}

\begin{theorem}\label{thm:construction}
Let $ m \geq 2 $ be a natural number and $ \pi: [m] \to [m]$ a bijection.
The hyperplane
\begin{equation}
\mathcal{H}_{\pi} \coloneqq 
\left\{ 
\mathbf{x} + \mathbf{y} \in \mathbb{R}^{m} \oplus \mathbb{R}^{p} \;\middle|\; 
\sum_{i=1}^m \pi(i) x_{i} =
\sum_{j=1}^p s_j y_j
\right\},
\end{equation}
intersects the $ (m+p) $-dimensional polytope $ \mathcal{P} = \operatorname{ConvexHull} \left(\mathscr{V}_{m} \cup \mathscr{W}_{m} \right) $ in a simplicial facet.
Therefore, $ \mathcal{P} $ has at least $ m! $ facets.
\end{theorem}

\begin{proof}
We first show that every vector in $ \mathscr{V}_{m} \cup \mathscr{W}_{m} $ satisfies the inequality
\begin{equation}
\sum_{i=1}^m \pi(i) x_{i} \geq
\sum_{j=1}^p s_j y_j.
\end{equation}
For any such vector the left-hand side consists of the sum of elements in a subset $ \mathrm{A} \subseteq [m] $.
The right-hand side is the triangular number $ t_{k} $ with $ k = \#\mathrm{A} $, and that is indeed the least value that the sum of $ k $ numbers from $ 1 $ to $ m $ can have.

Equality can only be achieved if the set $ \mathrm{A} $ corresponds to the smallest coefficients.
In other words, $ \mathrm{A} $ is either empty or one of
\[
	\mathrm{B}_k \coloneqq \left\{ \pi^{-1}(1), \dots, \pi^{-1}(k)\right\}, \quad k\in [m].
\]
The set $ \mathscr{U}_{m} \coloneqq \left( \mathscr{V}_{m} \cup \mathscr{W}_{m} \right) \cap \mathcal{H}_{\pi} $ of vectors in the hyperplane has $ m+1 $ points in $ \mathscr{V}_{m} $ and $ p-1 $ in $ \mathscr{W}_{m}$,
for a total of $ m+p $ points.
We now show that these points are affinely independent.
Since one of the points is the origin, this is equivalent to the fact that the nonzero vectors linearly span $ \mathcal{H}_{\pi} $.

For each $ 2\leq j \leq p $ we have a super triangular number $ s_{j} $ that is also the triangular number $ t_{k_j} $.
The vectors $ \mathbf{e}_{\mathrm{B}_{k_j}} + \mathbf{f}_{\mathrm{R}_{k_j}} $ and $ \mathbf{e}_{\mathrm{B}_{k_j}} + \mathbf{f}_{j} $
are both in $\mathscr{U}_{m}$.
Their difference is $ \mathbf{d}_{j} = \mathbf{f}_{j} - \mathbf{f}_{\mathrm{R}_{k_j}}  $.
The $ p \times (p-1) $ matrix having $ \mathbf{d}_{j} $ in column $ j $, restricted to its last $p-1$ rows, is unitriangular; hence the columns are linearly independent, and thus their linear span $ \mathsf{L} $ is $ (p-1) $-dimensional and contained in $ \mathbb{R}^{p} $.

On the other hand, the orthogonal projection of the nonzero vectors $ \mathscr{U}_{m} $ onto $ \mathbb{R}^{m} $ yields the set of vectors
\begin{equation}
	\left\{ \mathbf{e}_{\mathrm{B}_{1}}, \mathbf{e}_{\mathrm{B}_{2}}, \dots, \mathbf{e}_{\mathrm{B}_{m}}   \right\},
\end{equation}
which form a basis for $ \mathbb{R}^m $.
Since the linear subspace $ \mathsf{L} $ has a trivial projection to $ \mathbb{R}^{m} $, we can conclude that the linear span of $ \mathscr{U}_{m} $ is at least $ m + (p-1) $-dimensional.
This linear span is then equal to the hyperplane $ \mathcal{H}_{\pi} $.

As it is clear that the polytope $ \mathcal{P} $ has points which are not on the hyperplane, it follows that it is $ (m+p) $-dimensional and $ \mathcal{H}_{\pi} $ defines a simplicial facet.
As all of these hyperplanes are different, we have at least $ m! $ facets.
\end{proof}

Now we can prove our main result.

\newtheorem*{thm:main-body}{Theorem~\ref{thm:main}}
\begin{thm:main-body}
{\itshape 
    For every $n\geq 10$ there exists an $n$-dimensional $0/1$-polytope $\mathcal{P}\subseteq \mathbb{R}^n$ whose number of facets is at least $(n-\lceil 2\log_2 n\rceil - 1)!$.}
\end{thm:main-body}

\begin{proof}

The main construction produces a polytope of arbitrarily large dimension.
We now adapt the construction so that it covers all dimensions at least $10$. 
For $n=10$, the result follows from the unit cube. 
Hence, assume $n\geq 11$. Choose $p = \lceil 2\log_2 n\rceil + 1$ and $m = n-p$.
We have that $m\geq 2$ when $n\geq 10$.

Let $p'$ be the auxiliary parameter to $m$ in the construction.
We have by \eqref{eq:p_bound}
\[
	p' \leq \left\lceil \log_2\bigl(2m(m+1)\bigr)\right\rceil \leq \left\lceil \log_2(2n^{2})\right\rceil = p,
\]
since $m+1\leq n$.

Now Theorem \ref{thm:construction} provides an $(m+p')$-dimensional $0/1$-polytope $\mathcal{P}$ with at least $m!$ facets, where $m+p' \leq m+p = n$.

Since pyramids preserve the $0/1$ property, by repeatedly taking pyramids over $\mathcal{P}$ we can get a polytope $\mathcal{P}'$ of dimension $n$ with at least $m!=(n-p)!$ facets.
\end{proof}

\section{Asymptotics}

We can now combine the known upper bounds with the lower bound given by the construction. This yields the correct asymptotic behavior of $\log g(n)$, up to an error of $O((\log n)^2)$.

\begin{theorem}
    Let $g(n)$ be the maximum number of facets that an $n$-dimensional $0/1$-polytope can have.
    Then
    \[\log g(n) = n\log n - n + O\big((\log n)^2\big).\]
\end{theorem}

\begin{proof}
Using the construction in the proof of Theorem \ref{thm:main}, we obtain
\[
g(n) \geq (n-f(n))!, \qquad f(n)=\lceil2\log_2 n\rceil+1=O(\log n).
\]

Taking logarithms and using Stirling's approximation\footnote{In the form $\log(n!) = n\log n - n + O(\log n)$.}, we obtain
\[
\log g(n) \geq (n-f(n))\log(n-f(n)) - n + f(n) + O\bigl(\log(n-f(n))\bigr).
\]
Expanding $\log(n-f(n)) = \log n - f(n)/n + O\bigl(f(n)^2/n^2\bigr)$, using that $f(n)=\frac{2}{\log 2}\,\log n+O(1)$, and collecting the error terms, all dominated by $O(\log n)$, we arrive at

\[
    \log g(n) \geq n\log n - n - \frac{2}{\log 2}(\log n)^2 + O(\log n). 
\]

On the other hand, the known upper bound of $g(n)\leq 30(n-2)!$ gives
\[
\log g(n) \leq n\log n-n+O(\log n).
\]
Together with the preceding lower bound, this shows that $\log g(n) = n\log n - n + O((\log n)^2)$.
\end{proof}

\section{Final remarks}

Our theorem determines the asymptotics of $\log g(n)$, but leaves open how far $g(n)$ is from $n!$. Let us define
\[
    D(n)\coloneqq\log\left(\frac{n!}{g(n)}\right).
\]
Our main result and the upper bounds in \cite{fleiner2000upper} give
\[
    2\log n-O(1)
    \leq D(n)
    \leq \frac{2}{\log 2}(\log n)^2+O(\log n).
\]
(Note that the upper bound in the above display comes from our Theorem~\ref{thm:main}, and the lower bound comes from \cite[Corollary~8]{fleiner2000upper}.)
The natural question is whether $D(n)$ has order $\log n$ or $(\log n)^2$. We believe the latter, i.e., in a certain sense our lower bound for $g(n)$ is closer to the true value of $g(n)$ than the upper bound in \cite{fleiner2000upper}.

A few comments are in order. The factor $2$ in our construction should not be taken too seriously; we believe it can be improved. Preserving the notation of the main proof, the $m!$ special facets in our construction come from the maximal chains through the $m+1$ hypersimplex layers. Any binary lift of this type needs about $\log_2 m$ auxiliary coordinates only to distinguish the layers; the additional points needed to span the auxiliary directions do not appear especially relevant. Our construction in the present paper spends roughly twice $\log_2 m$ as it represents every integer up to $t_m=\Theta(m^2)$, although only $m+1$ heights are used. Using the present greedy strategy this factor $2$ is unavoidable, since $S_p\leq2^{p-1}$. Removing it would require a different lifting.

At the other end, the factorial upper bound in \eqref{eq:upper-bound} counts distinct integral normal vectors under a global $\ell_1$ budget \cite{fleiner2000upper}. This retains the size of the normal vectors but washes out most of the compatibility needed for all of them to support facets of a single $0/1$-polytope. To us, $(n-2)!$ looks more like the outcome of a generous counting argument than the shadow of an extremal example. We dare to formulate the following deliberately provocative conjecture.

\begin{conjecture}\label{conj:factorial-defect}
There exists an absolute constant $C>0$ such that
\[
    \lim_{n\to\infty}
    \frac{\log(n!)-\log g(n)}{(\log n)^2}
    =C.
\]
Equivalently,
\[
    g(n)
    =
    n!\exp\left(-(C+o(1))(\log n)^2\right)
    =
    n!\,n^{-(C+o(1))\log n}.
\]
\end{conjecture}

Roughly speaking, we conjecture that $g(n)$ behaves more like $(n-\lceil2\log_2 n\rceil-1)!$ than $(n-2)!$. A completely different construction could prove us wrong, but we would be quite surprised. Even the weaker estimate $D(n)=\omega(\log n)$ would already rule out $(n-c)!$ as the correct scale for any fixed integer $c$.

\bibliographystyle{amsalpha}
\bibliography{bibliography}

\end{document}